\documentclass[11pt, reqno]{amsart}

\usepackage{amssymb,amsmath,amsthm,amsfonts,amscd,bbm}
\usepackage{enumitem}
\usepackage{pdflscape}
\usepackage{caption}
\usepackage{bm}

\usepackage{ifpdf}
\ifpdf
\usepackage[pdftex]{graphicx}
\else
\ufinite-dimensionalraphicx}
\fi
\usepackage[all]{xy}
\usepackage{hyperref}
\usepackage{xcolor}
\usepackage{graphicx}
\usepackage[outdir=./]{epstopdf}
\usepackage{enumitem}
\usepackage{tensor}
\usepackage{tikz-cd}
\usepackage{multicol}
\usepackage{mathtools}
\usepackage{tocvsec2}
\usepackage{bbm}
\usepackage{longtable}
\usepackage{fullpage}

\newcommand{\doi}[1]{\href{https://dx.doi.org/#1}{{\small DOI:#1}}}

\newcommand{\googlebooks}[1]{(preview at \href{https://books.google.com/books?id=#1}{google books})}

\newcommand{\numdam}[1]{}

\usepackage{mathrsfs}
\DeclareMathAlphabet{\mathpzc}{OT1}{pzc}{m}{it}

\def\semicolon{;}
\def\applytolist#1{
    \expandafter\def\csname multi#1\endcsname##1{
        \def\multiack{##1}\ifx\multiack\semicolon
            \def\next{\relax}
        \else
            \csname #1\endcsname{##1}
            \def\next{\csname multi#1\endcsname}
        \fi
        \next}
    \csname multi#1\endcsname}

\def\calc#1{\expandafter\def\csname c#1\endcsname{{\mathcal #1}}}
\applytolist{calc}QWERTYUIOPLKJHGFDSAZXCVBNM;
\def\bbc#1{\expandafter\def\csname bb#1\endcsname{{\mathbb #1}}}
\applytolist{bbc}QWERTYUIOPLKJHGFDSAZXCVBNM;
\def\bfc#1{\expandafter\def\csname bf#1\endcsname{{\mathbf #1}}}
\applytolist{bfc}QWERTYUIOPLKJHGFDSAZXCVBNM;
\def\sfc#1{\expandafter\def\csname s#1\endcsname{{\sf #1}}}
\applytolist{sfc}QWERTYUIOPLKJHGFDSAZXCVBNM;
\def\fc#1{\expandafter\def\csname f#1\endcsname{{\mathfrak #1}}}
\applytolist{fc}QWERTYUIOPLKJHGFDSAZXCVBNM;

\usepackage{tikz}
\usepackage{tikz-cd}

\makeatletter
\def\fixtikzforbreqn#1#2{%
  \protected\edef#1{\noexpand\ifmmode\mathchar\the\mathcode`#2 \noexpand\else#2\noexpand\fi}%
}
\fixtikzforbreqn\tikz@nonactivesemicolon;
\fixtikzforbreqn\tikz@nonactivecolon:
\fixtikzforbreqn\tikz@nonactivebar|
\fixtikzforbreqn\tikz@nonactiveexlmark!
\makeatother

\usepackage{breqn}   % Fix : | ; ! in tikzcd. And must put makealetter-other between tikzcd package and breqn package

\usetikzlibrary{arrows,backgrounds,patterns.meta}
\usetikzlibrary{positioning,shadings,cd}
\usetikzlibrary{shapes}
\usetikzlibrary{backgrounds}
\usetikzlibrary{decorations,decorations.pathreplacing,decorations.markings,decorations.pathmorphing}
\usetikzlibrary{fit,calc,through}
\usetikzlibrary{external}
\usetikzlibrary{arrows}
\tikzset{vertex/.style = {shape=circle,draw,fill=black,inner sep=0pt,minimum size=5pt}}
\tikzset{edge/.style = {->,> = latex', bend right}}
\tikzset{
	super thick/.style={line width=3pt}
}
\tikzset{
    quadruple/.style args={[#1] in [#2] in [#3] in [#4]}{
        #1,preaction={preaction={preaction={draw,#4},draw,#3}, draw,#2}
    }
}
\tikzstyle{shaded}=[fill=red!10!blue!20!gray!30!white]
\tikzstyle{unshaded}=[fill=white]
\tikzstyle{empty box}=[circle, draw, thick, fill=white, opaque, inner sep=2mm]
\tikzstyle{annular}=[scale=.7, inner sep=1mm, baseline]
\tikzstyle{rectangular}=[scale=.75, inner sep=1mm, baseline=-.1cm]
\tikzstyle{mid>}=[decoration={markings, mark=at position 0.5 with {\arrow{>}}}, postaction={decorate}]
\tikzstyle{mid<}=[decoration={markings, mark=at position 0.5 with {\arrow{<}}}, postaction={decorate}]
\tikzstyle{over}=[double, draw=white, super thick, double=]
\tikzstyle{snake}=[decorate, decoration={snake, segment length=1mm, amplitude=.3mm}]
\tikzstyle{saw}=[decorate, decoration={saw, segment length=.7mm, amplitude=.25mm}]

\tikzstyle{coupon}=[draw, very thick, rectangle, rounded corners=5pt]
\tikzset{Rightarrow/.style={double equal sign distance,>={Implies},->},
triplecd/.style={-,preaction={draw,Rightarrow}},
quadruplecd/.style={preaction={draw,Rightarrow,
shorten >=0pt
},
shorten >=1pt,
-,double,double
distance=0.2pt}}
\tikzset{
    tripleline/.style args={[#1] in [#2] in [#3]}{
        #1,preaction={preaction={draw,#3},draw,#2}
    }
}
\tikzstyle{triple}=[tripleline={[line width=.15mm,black] in
      [line width=.7mm,white] in
      [line width=1mm,black]}] 
\tikzset{
    quadrupleline/.style args={[#1] in [#2] in [#3] in [#4]}{
        #1,preaction={preaction={preaction={draw,#4},draw,#3}, draw,#2}
    }
}
\tikzstyle{quadruple}=[quadrupleline={[line width=.3mm,white] in
      [line width=.6mm,black] in
      [line width=1.2mm,white] in
      [line width=1.5mm,black]}]

\theoremstyle{plain}
\newtheorem{thm}{Theorem}[section]
\newtheorem*{thm*}{Theorem}

\newtheorem{cor}[thm]{Corollary}

\newtheorem*{cor*}{Corollary}

\newtheorem*{conj*}{Conjecture}
\newtheorem{lem}[thm]{Lemma}
\newtheorem*{lem*}{Lemma}

\newtheorem{prop}[thm]{Proposition}

\newtheorem*{quest*}{Question}
\newtheorem*{claim*}{Claim}

\theoremstyle{definition}
\newtheorem{defn}[thm]{Definition}
\newtheorem{cons}[thm]{Construction}

\newtheorem{sub-ex}[thm]{Sub-Example}
\newtheorem{counter-ex}[thm]{Counter-Example}
\newtheorem{rem}[thm]{Remark}
\newtheorem*{rem*}{Remark}
\newtheorem{remark}[thm]{Remark}

\usepackage{xcolor}
\definecolor{dark-red}{rgb}{0.7,0.25,0.25}
\definecolor{dark-blue}{rgb}{0.15,0.15,0.55}
\definecolor{medium-blue}{rgb}{0,0,.8}
\definecolor{DarkGreen}{RGB}{0,150,0}
\definecolor{rho}{named}{red}
\hypersetup{
   colorlinks, linkcolor={purple},
   citecolor={medium-blue}, urlcolor={medium-blue}
}

\newcommand{\Hilb}{\mathsf{Hilb}}

\newcommand{\rCorr}{{\mathsf{C^*Alg}}}

\def\altdb{\vadjust{\vbox to 0pt{\vss\hbox{\kern \hsize
\quad{\dbend}}\kern\baselineskip\kern-10pt}}}

\newcommand{\noshow}[1]{}
\renewcommand{\MR}[1]{}

\begin{document} 
\title{Quantum symmetries of noncommutative tori }
\author{David E. Evans, Corey Jones}
\address{David E Evans, School of Mathematics, Cardiff University, Senghennydd Road, Cardiff CF24 4AG, Wales, United Kingdom}
\address{Corey Jones, Department of Mathematics,
North Carolina State University, Raleigh, NC 27695, USA}
\dedicatory{This paper is dedicated to the memory of Huzihiro Araki}

\begin{abstract}
We consider the problem of building non-invertible quantum symmetries (as characterized by actions of unitary fusion categories) on noncommutative tori. We introduce a general method to construct actions of fusion categories on inductive limit C*-algberas using finite dimenionsal data, and then apply it to obtain AT-actions of arbitrary Haagerup-Izumi categories on noncommutative 2-tori, of the even part of the $E_{8}$ subfactor on a noncommutative 3-torus, and of $\text{PSU}(2)_{15}$ on a noncommutative 4-torus.

\end{abstract}

\maketitle

\section{Introduction}

Fusion categories are algebraic objects that characterize generalized finite symmetries in many areas of quantum physics and mathematics. They are the structures underlying the standard invariants of finite depth subfactors, and characterize the algebraic structure of superselection sectors in quantum field theory. In both of these contexts, a fusion category ``acts by bimodules" on a von Neumann (or C*) algebra $A$, generalizing the notion of a (cocycle) group action. This provides a framework for non-invertible quantum symmetries. More precisely, an \textit{action} of a unitary fusion category $\mathcal{C}$ on a C*-algebra $A$ is a C*-tensor functor from $\mathcal{C}$ to the C*-tensor category $\text{Bim}(A)$ of bimodules.\footnote{by bimodule of C*-algebra we mean a correspondence and we also assume our C*-algebras are unital} 

This perspective on quantum symmetries originated in Vaughan Jones' theory of subfactors \cite{MR0696688}. Popa's fundamental results on subfactors translate to a complete classification of fusion category actions on amenable factors \cite{MR1055708, MR1339767}. Much less is known for fusion category actions on simple amenable C*-algebras, and the situation is expected to be significantly more complicated. Recently, however, there has been significant interest in this direction \cite{MR1900138, chen2022ktheoretic, arano2023tensor, pacheco2023elliott, 2105.05587, izumi2024gkernels, pacheco2023classification} (see the introduction to \cite{gabe2023dynamical} for a guide to recent progress on the closely related problem of classifying group actions on amenable C*-algebras). Using the machinery of subfactor theory, it is fairly straightforward to construct (and even classify in some instances) a wide variety of actions on AF-algebras \cite{MR996454, MR1642584, MR1055708, chen2022ktheoretic} and Cuntz-Kreiger algebras \cite{MR1604162, MR1832764}. However, there appears to be a lack of examples in the literature beyond these cases.

A natural family of stably finite amenable C*-algebras to consider beyond the AF case are noncommutative tori. Given a skew-symmetric real $n\times n$ matrix $\Theta$, $A_{\Theta}$ is the universal C*-algebra generated by $n$ unitaries $U_{1},\dots U_{n}$ satisfying $U_{k}U_{j}=e^{2\pi i \Theta_{j,k}}U_{j}U_{k}$. This is the algebra of continuous functions on the noncommutative n-torus. If the $\mathbbm{R}^{\times}$ valued bicharacter $b_{\Theta}(x,y):=e^{2\pi i \langle x, \Theta y\rangle}$ on $\mathbbm{Z}^{n}$ is nondegenerate, then $A_{\Theta}$ is a simple nuclear AT C*-algebra with unique trace \cite{phillips2006simple}. The algebras $A_{\Theta}$ (and their smooth subalgebras) have wide-ranging applications in both condensed matter \cite{MR1295473, MR0862832} and high-energy physics \cite{MR1613978, RevModPhys.73.977}, and are the paradigmatic example in Connes' noncommutative geometry \cite{MR1303779}.

The motivation for this paper is to address the problem of finding genuine ``quantum" symmetries of noncommutative tori. The canonical examples of fusion categories beyond the group theoretical setting are the fusion categories $SU(2)_{k}$, which play an important role in subfactor theory and conformal field theory. Our investigation starts with a no-go result, which uses elementary ordered K-theory and a bit of number theory: 

\begin{thm} The fusion categories $SU(2)_{k}$ with $k+2$ prime and $k+1$ not a power of $2$ admit no action on any noncommutative torus of any rank.
\end{thm}

This result demonstrates (perhaps unsurprisingly) that noncommutative tori are not universal for finite quantum symmetries. Obtaining positive results about the existence of actions is more difficult, and as the above no-go result demonstrates, will necessarily depend on the details of the fusion category in question.

The basic idea we exploit in this paper comes from the results of \cite{MR1247990, phillips2006simple}, which demonstrate that simple non-commutative tori are inductive limits of C*-algebras of the form $C(\mathbbm{T}, A)$, where $A$ is a finite dimensional C*-algebra. In other words, simple noncommutative tori are \textit{AT-algebras}. Our strategy consists of building actions of fusion categories on ``building block" C*-algebras $C(\mathbbm{T}, A)$, and then choosing sufficiently nice connecting maps between these blocks to build an inductive limit action on an AT-algebra, which can be identified with a noncommutative torus via general $K$-theoretic classification results for AT-algebras. In particular, Elliott \cite{MR1462852, MR1403992} gave a complete classification of simple AT-algebras through the ordered $K_0$-group, $K_1$, the simplex of tracial states and its pairing with $K_0$, which we can apply in our setting.
In general, we call actions built from circle algebras in this way \textit{AT-actions} \ref{def:AF-AT-actions}. The technical result we will use to produce our actions is summarized in Theorem \ref{mainextension}, which gives a construction of an AT action from some basic categorical data. 

Our methods are very general and can be easily applied to build actions of fusion categories on a wide range of amenable C*-algebras, realized as inductive limits of elementary building blocks. In this paper, we apply our framework to obtain the following existence results for quantum symmetries on noncommutative tori:

\bigskip

\begin{enumerate}
\item
(Section \ref{sec:HaagIz}). Any Haagerup-Izumi fusion category $\mathcal{H}$ over any abelian group $G$ admits an AT-action on a simple noncommutative 2-torus.

\bigskip

\item 
(Section \ref{sec:AdE8}). The even part of an $\text{E}_{8}$ subfactor admits an AT-action on a simple noncommutative 3-torus.

\bigskip

\item 
(Section \ref{sec:PSU(2)}). The fusion category $\text{PSU}(2)_{15}$ admits an AT-action on a simple noncommutative 4-torus.
\end{enumerate}

In each of the cases mentioned above, we actually produce two a-priori different actions of each of the fusion categories on noncommutative tori (see Remark \ref{rem:differentactions}). One is manifestly an AT-action, but for the other it is not clear if it is AT. Determining whether these actions are equivalent is an interesting test case for any proposed classification of fusion category actions beyond the AF-setting.

One of the main motivations for considering actions of fusion categories on noncommutative tori is the ubiquity of these algebras in applications of operator algebras to mathematical physics, as mentioned above. Our hope is that we can give concrete applications of fusion categorical symmetries in these situations. We plan to pursue this in future work.

\section{Actions of unitary tensor categories on C*-algebras}

Recall that a unitary tensor category is a rigid C*-tensor category with simple unit object (our C*-tensor categories are always assumed to be unitarily Cauchy complete, i.e. closed under taking orthogonal direct sums and summands). A \textit{unitary fusion category} is a unitary tensor category with finitely many isomorphism classes of simple objects. Standard examples are $\text{Rep}(\mathbbm{G})$ where $\mathbbm{G}$ is a finite (quantum) group, and the categories $\mathcal{C}(\mathfrak{g},k)$, where $\mathfrak{g}$ is a simple Lie algebra, $k\in\mathbbm{N}$, which arise in the contexts of conformal field theory and Drinfeld-Jimbo quantum groups (see \cite{MR4642115, MR4079742} for overviews). In this paper, we typically use the ``loop group" notation and identify $\text{SU}(2)_{k}=\mathcal{C}(\mathfrak{sl}_{2},k)$.

Associated to a unital C*-algebra $A$ is a C*-tensor category $\text{Bim}(A)$. It consists of all (right) $A$-$A$ C*-correspondences, which are finitely generated projective as right Hilbert modules. The relative tensor product is denoted $\boxtimes_{A}$. See \cite[Section 2.2]{MR4419534} and references therein for definitions and discussion related to C*-correspondences.

\begin{defn}
An action of a unitary fusion category $\cC$ on a unital C*-algebra is a C*-tensor functor $F: \cC\rightarrow \text{Bim}(A)$. 
\end{defn}

We will often use the notation We use the notation $\cC\overset{F}{\curvearrowright} A$. Actions of fusion categories generalize (cocyle) group actions. Indeed, any $G$ action on a C*-algebra yields an action of the fusion category $\text{Hilb}(G)$ of G-graded finite dimensional Hilbert spaces on $A$.

\begin{rem}\label{finitedimensionalaction}Actions of fusion categories on finite dimensional C*-algebras can be completely understood in terms of their module categories via an equivalence of 2-categories. We will sketch how this works at the level of objects, but for a more detailed explanation, see \cite[Proposition 3.4]{chen2022ktheoretic}).

Let $\mathcal{C}$ be a unitary fusion category.

\begin{itemize}

\item 
Given a (right) unitary, finitely semisimple module category $\mathcal{M}$, pick a progenerator (or in other words, $m\in \mathcal{M}$ which contains every simple object in $\mathcal{M}$ up to isomorphism. Then $A:=\text{End}(m)$ is a finite dimensional C*-algebra and $\mathcal{M}\cong \text{Rep}(A)$, where the latter denotes the C*-category of (right) finite dimensional Hilbert space representations of $A$. Thus the C*-tensor category of $*$-endofunctors $\text{Bim}(A)^{op}\cong \text{End}(\mathcal{M})$ with $X\mapsto \cdot \boxtimes_{A} X $. Hence a right module category gives rise to an action on the finite dimensional algebra $A$.

\item 
Conversely, given an action $F:\mathcal{C}\rightarrow \text{Bim}(A)$, the category $\mathcal{M}:=\text{Rep}(A)$ is a finitely semisimple, C*-category. The equivalence $\text{Bim}(A)^{op}\cong \text{End}(\mathcal{M})$ as above leads to a right $\mathcal{C}$-module structure on $\mathcal{M}$. This construction is inverse to the one above.

\item For a C$^*$-algebra $A$, an $A$-$A$ bimodule of finite Jones index type yields an element of the Kasparov group $KK(A,A)$ and so an element of $\text{End}(K_0(A))$ \cite{MR1624182}. Hence an action of $\mathcal{C}$ on $A$, yields a unital homomorphism from $\text{Fus}(\mathcal{C})$ to $\text{End}(K_0(A))$.

 \end{itemize}

Typically, when a new class of fusion categories is discovered, one of the first problems to consider is a characterization of all their module categories. This has resulted in a fairly comprehensive understanding of all actions of fusion categories on finite-dimensional C*-algebras.
\end{rem}

%%%%%%%%%%%%%%%%%%%%%%%%%%%%%%%

\subsection{Inductive limit actions}

One way to build an action of a fusion category $\cC$ on a C*-algebra $A$ is to realize $A$ as the inductive limit of a sequence of C*-algebras $A_{1}\rightarrow_{\iota_{1}} A_{2}\rightarrow_{\iota_{2}} \cdots \rightarrow A$, and build actions at each level. We then equip the connecting maps with ``$\cC$-equivariant" structures. This will yield an action of $\cC$ on $A$ as in \cite[Proposition 4.4]{chen2022ktheoretic}, \cite[Section 4]{pacheco2023elliott}. 

When each level in an inductive limit decomposition is finite-dimensional, the resulting algebras are AF-algebras and the resulting actions are called \textit{AF-actions}. Given a unitary fusion category, these are all realized using the combinatorial data of $\cC$ and its module categories (as explained in the previous section) via a generalized Bratteli diagram \cite[Definition 4.10]{chen2022ktheoretic}. Furthermore, such actions can be completely classified using K-theoretic invariants, which is the main result in \cite{chen2022ktheoretic}.

Our goal is to use inductive limit approximations to build actions on algebras beyond the class of AF-algebras, but using the AF-actions as a starting point. Indeed, all classifiable C*-algebras can be realized as inductive limits of C*-algebras with relatively simple building blocks \cite{MR1241132, MR1661611}. This is also true for many cases of group actions as well \cite[Section 6]{MR1053492}, \cite[Section 4]{MR3572256} and so it is conceivable that this approach will lead to a construction of all possible actions (at least, that are ``classifiable" in some yet-to-be-determined sense).

The following gives a suitable notion of ``equivariant homomorphism" between unital C*-algebras equipped with a $\mathcal{C}$-action.

\begin{defn}\label{EquivariantMorphism} \cite[Lemma 3.8]{chen2022ktheoretic} Let $\cC\overset{F}{\curvearrowright} A$ and $\cC\overset{G}{\curvearrowright} B$ be actions of a fusion category on unital $\rm C^*$-algebras,
and let $\iota: A\rightarrow B$ be an injective, unital $\rm C^*$-algebra homomorphisms. An equivariant structure on $\iota$ consists of a family of linear maps 
$$\{j^{X}: F(X)\rightarrow G(X)\}_{X\in\cC}$$ such that 
\begin{enumerate}[label=(\arabic*)]
\item $j^{X}(a\rhd x \lhd a')=\iota(a)\rhd j^{X}(x)\lhd \iota(a')$ for $a,a'\in A$ and $x\in F(X)$.
\item
For any morphism $f\in \cC(X\rightarrow Y)$, $G(f)\circ j^{X}=j^{Y}\circ F(f)$;
\item
$\iota(\langle x | y\rangle_A)=\langle j^{X}(x) | j^{X}(y)\rangle_B$ for $x,y\in F(X)$;
\item
$j^{X}(F(X))B=G(X)$;
\item The following diagram commutes:
\[
\begin{tikzcd}
F(X)\otimes_{\mathbbm{C}} F(Y)
\arrow[swap]{d}{j^{X}\otimes j^{Y}}
\arrow{r}{}
& F(X)\boxtimes_{A} F(Y)
\arrow{r}{\mu^{F}_{X,Y}}
& F(X\otimes Y)\arrow{d}{j^{X\otimes Y}}
 \\
G(X)\otimes_{\mathbbm{C}}G(Y)
\arrow{r}{}
& G(X)\boxtimes_{B} G(Y)
\arrow{r}{\mu^{G}_{X,Y}}
& G(X\otimes Y)
\end{tikzcd}
\]

For $v\in F(X), w\in F(Y)$, this reads $j^{X\otimes Y}(\mu^{F}_{X,Y}(v\boxtimes_{A} w))=\mu^{G}_{X,Y}(j^{X}(v)\boxtimes_{B} j^{Y}(w))$
\end{enumerate}

\end{defn}

\bigskip

\begin{defn}  Let $\cC\overset{F}{\curvearrowright} A$ and $\cC\overset{G}{\curvearrowright} B$ be actions, $\iota:A\rightarrow B$ a unital $*$-homomorphism, and $\{j^{X}:F(X)\rightarrow G(X)\}$ and equivariant structure on $\iota$. Define 

$$B^{\cC}_{\iota}:=\{b\in B\ :\ \text{for all}\ X\in \mathcal{C},\ x\in X,\ b\triangleright j^{X}(x)=j^{X}(x)\triangleleft b \}.$$
 \end{defn}

\bigskip

\noindent Note that $B^{\cC}_{\iota}\subseteq B$ is a unital subalgebra and that $[B^{\cC}, \iota(A)]=0$. Indeed, we will think of this as a generalized ``relative commutant" of the whole included action $j(\mathcal{C})$. 

Now we describe the fundamental construction we use to pass from a $\mathcal{C}$ action on a finite-dimensional C*-algebra $A$ to an action on $A\otimes C$ for an arbitrary (unital) C*-algebra $C$ without any extra data.

\begin{defn} Given an action $\cC\overset{F}{\curvearrowright} A$ and any unital  C*-algebra $C$, we define a new  action $\cC\overset{\widetilde{F}}{\curvearrowright} A\otimes C$ called the \textit{tensor extension} of $F$ by $C$ by setting

$$\widetilde{F}(X):=F(X)\otimes C.$$

\noindent Here, the bimodule structure is specified (on simple tensors) by

$$(a\otimes c)\triangleright (v\otimes d) \triangleleft (a^{\prime}\otimes c^{\prime}):=(a\triangleright v \triangleleft a^{\prime})\otimes cdc^{\prime},$$

\noindent with right $A\otimes C$ valued inner product

$$\langle v\otimes c \ | \  w\otimes d\rangle_{A\otimes C}:=\langle v\ |\ w\rangle_{A}\otimes c^{*}d.$$

This is positive definite on the algebraic tensor product, and thus we complete $\widetilde{F}(X)$ in the induced norm. In this paper, either $F(X)$ or $C$ will be finite-dimensional and thus no completion is required. Furthermore, for $f\in \mathcal{C}(X,Y)$, we can define

$$\widetilde{F}(f)(v\otimes c):=F(f)(v)\otimes c$$

\noindent for $v\in X$. It is easy to see that $\widetilde{F}$ gives a functor from $\cC$ to $\text{Bim}(A\otimes C)$.

We use the obvious structure morphisms to build the monoidal structure on $\widetilde{F}$, 

$$\mu^{\widetilde{F}}_{X,Y}((v\otimes c)\boxtimes_{A\otimes C} (w\otimes d)):=\mu^{F}_{X,Y}(v\boxtimes_{A} w)\otimes cd $$

\noindent This is well-defined and satisfies the required coherences. 

\end{defn}

\bigskip

The next lemma is a technical tool for extending actions on finite-dimensional building blocks to circle algebras.

\begin{lem}\label{mainlemma}
Let $\iota:A\rightarrow B$ be a unital inclusion of C*-algebras, $\cC\overset{F}{\curvearrowright} A$ and $\cC\overset{G}{\curvearrowright} B$ be actions, and let $\{j^{X}:F(X)\rightarrow G(X)\}_{X\in \cC}$ be an equivariant structure on $\iota$ with respect to these actions. Let $C$ be a unital C*-algebra, and $\nu:A\otimes C\rightarrow B\otimes C$ a unital inclusion such that 

\begin{enumerate}
\item
$\iota\otimes \text{id}_{C}=\nu|_{A\otimes {C}}$
\item
$\nu(1\otimes C)\subseteq B^{\cC}_{\iota}\otimes C$
\end{enumerate}  Then there exists a canonical equivariant structure $\{k^{X}:\widetilde{F}(X)\rightarrow \widetilde{G}(X)\}_{X\in \cC}$ on $\nu$ with respect to the tensor extensions of $F$ and $G$.
\end{lem}

\begin{proof}

We define $k^{X}:\widetilde{F}(X)\rightarrow \widetilde{G}(X)$ for a simple tensor $v\otimes c\in F(X)\otimes C$ by

$$k^{X}(v\otimes c)=(j^{X}(v)\otimes 1_{C})\triangleleft \nu(1\otimes c).$$

First we check the isometry property, which will guarantee $k^{X}$ extends to a isometry on all of $\widetilde{F}(X)$:

\begin{align*}
\langle k^{X}(v\otimes c)\ |\  k^{X}(w\otimes d)\rangle_{B\otimes C}&=\langle (j^{X}(v)\otimes 1_{C})\triangleleft \nu(1\otimes c)\ |\  (j^{X}(w)\otimes 1_{C})\triangleleft \nu(1\otimes d)\rangle_{B\otimes C}\\
&=\nu(1\otimes c)^{*}\left[\langle j^{X}(v)\ |\  j^{X}(w)\rangle_{B}\otimes 1_{C} \right] \nu(1\otimes d)\\
&=\nu(1\otimes c)^{*}(\iota(\langle v\ |\  w\rangle_{A})\otimes 1_{C} )\nu(1\otimes d)\\
&=\nu(1\otimes c^{*})(\nu(\langle v\ |\  w\rangle_{A}\otimes 1_{C} )\nu(1\otimes d)\\
&=\nu(\langle v\ |\  w\rangle_{A}\otimes c^{*}d)\\
&=\nu(\langle v\otimes c\ |\  w\otimes d\rangle_{A\otimes C}).
\end{align*}

Now we check this satisfies the compatibility with $\nu$ on simple tensors

\begin{align*}
k^{X}((a\otimes c) \triangleright (v\otimes d) \triangleleft (a^{\prime}\otimes c^{\prime}))&=(j^{X}(a\triangleright v\triangleleft a^{\prime})\otimes 1_{C})\triangleleft \nu(1\otimes cdc^{\prime})\\
&=\left(\iota(a)\triangleright j^{X}(v)\triangleleft \iota(a^{\prime})\otimes 1_{C}\right) \triangleleft \nu(1\otimes c)\nu(1\otimes d)\nu(1\otimes c^{\prime})\\
&=\left[\nu(a\otimes 1_{C})\triangleright (j^{X}(v)\otimes 1_{C})\triangleleft (\nu(a^{\prime}\otimes 1_{C})\right]\triangleleft \nu(1\otimes c)\nu(1\otimes d)\nu(1\otimes c^{\prime})\\
&=\nu(a\otimes c)\triangleright \left[(j^{X}(v)\otimes 1_{C}\right]\triangleleft \nu(1\otimes d))\triangleleft \nu(a^{\prime}\otimes c^{\prime})\\
&=\nu(a\otimes c)\triangleright k^{X}(v\otimes d)\triangleleft \nu(a^{\prime}\otimes c^{\prime}).
\end{align*}

\noindent In the third equality we have used the hypothesis that $\iota(a)\otimes 1_{C}=\nu(a\otimes 1_{C})$ and in the fourth equality we have used $\nu(1_{A}\otimes C)\subset B^{\mathcal{C}}_{\iota}$. For naturality, let $f\in \cC(X,Y)$. Then we compute for a simple tensor $v\otimes c\in \widetilde{F}(X)=F(X)\otimes C$

\begin{align*}
\widetilde{G}(f)(k^{X}(v\otimes c)&=(G(f)\otimes \text{id}_{C})\left[ (j^{X}(v)\otimes 1_{C})\nu(1_{A}\otimes c)  \right]\\
&= (G(f)j^{X}(v)\otimes 1_{C})\nu(1_{A}\otimes c)\\
&=(j^{Y}F(f)(v)\otimes 1_{C})\nu(1_{A}\otimes c)\\
&=k^{Y}(F(f)(v)\otimes c)\\
&=k^{Y}\left[\widetilde{F}(f)(v\otimes c)\right]
\end{align*}

Now for non-degeneracy, note that since $\nu$ is unital, $\nu(1_{A}\otimes C)B\otimes C=B\otimes C$, so

\begin{align*}
k^{X}(\widetilde{F}(X))(B\otimes C)&=(j^{X}(F(X))\otimes 1_{C})\nu(1_{A}\otimes C)(B\otimes C)\\
&=j^{X}(F(X))B\otimes C\\
&=G(X)\otimes C.
\end{align*}

Finally, it remains to check the commutativity of the diagram $(5)$. For $a=v\otimes c$ and $b=w\otimes d$

\begin{align*}
k^{X\otimes Y}(\mu^{\widetilde{F}}_{X,Y}(a\boxtimes_{A\otimes C} b)) &= k^{X\otimes Y}\left[\mu^{F}_{X,Y}(v\boxtimes_{A} w)\otimes cd\right]\\
&=\left[j^{X\otimes Y}(\mu^{F}_{X,Y}(v\boxtimes_{A} w))\otimes 1_{C}\right] \triangleleft \nu(1_{A}\otimes cd)\\
&=\mu^{G}_{X,Y}\left[j^{X}(v)\boxtimes_{B} j^{Y}(w)\otimes 1_{C}\right]\triangleleft \nu(1_{A}\otimes cd)\\
&=\mu^{\widetilde{G}}_{X,Y}\left(k^{X}(v\otimes c)\boxtimes_{B\otimes C} k^{Y}(w\otimes d)\right).
\end{align*}

\end{proof}

We now recall some definitions from \cite{chen2022ktheoretic} (c.f. \cite{pacheco2023elliott}).

\begin{defn}
Let $\cC$ be a fusion category and $A_{1}\rightarrow_{\iota_{1}} A_{2}\rightarrow_{\iota_{2}} \cdots $ be a tower of unital inclusions of C*-algebras. \textit{Inductive limit action data} consists of a family of actions $\cC\overset{F_{n}}{\curvearrowright} A_{n}$, and for each $n$, a $\mathcal{C}$-equivariant structure $j_{n}$ on the homomorphism $\iota_{n}:A_{n}\rightarrow A_{n+1}$.

\end{defn}

This is the same thing as a unitary tensor functor from $\cC$ to the 2-category $\textbf{IndLimC*-alg}$. We record the following proposition: 

\begin{prop}\cite[Proposition 4.4]{chen2022ktheoretic} Given inductive limit action data on a tower of unital inclusions $A_{1}\rightarrow_{\iota_{1}} A_{2}\rightarrow_{\iota_{2}} \cdots $,  there is an action $\cC\overset{F}{\curvearrowright} \varinjlim A_{n}$. 
\end{prop}

Setting $A=\varinjlim A_{n}$, the construction of the $A-A$ bimodules $F(X)$ closely parallels the construction of inductive limit C*-algebras. As many interesting C*-algebras arise as inductive limits of simpler building blocks, this gives a powerful tool for constructing actions of fusion categories on interesting C*-algebras. Up until now, this has mainly been applied to building actions of fusion categories on AF-algebras, using ideas from subfactor theory going back to V. Jones \cite{MR0696688}, Ocneanu \cite{MR996454}, and Popa \cite{MR1055708} (see \cite{MR1473221,MR1642584}).

There are two types of inductive limit actions we will be concerned with in this paper. 

\begin{defn}\label{def:AF-AT-actions}
    An inductive limit action is called 
    
    \begin{enumerate}
        \item 
    an \textit{AF-action} if each $A_{n}$ is finite dimensional. 
    \item 
     an \textit{AT-action} if each $A_{n}$ is finite dimensional algebra tensored with $C(\mathbbm{T})$.
    \end{enumerate}
\end{defn}

Note that an AF-action is necessarily on an AF-algebra, and an AT-action is on an AT-algebra, but not all actions on AF or AT algebras are AF or AT. However, the following corollary gives us an explicit tool for extending AF-actions to AT-actions.

\begin{cor}\label{maincor}
Suppose we have an inductive limit sequence $A_{1}\rightarrow_{\iota_{1}} A_{2}\rightarrow_{\iota_{2}} \cdots $ with actions $\cC\overset{F_{n}}{\curvearrowright} A_{n}$ and $\cC$-equivariant structures $j_{n}$ for each $\iota_{n}$. Further, let $C$ be a unital C*-algebra and suppose we have a sequence of homomorphisms $\nu_{n}: A_{n}\otimes C\rightarrow A_{n+1}\otimes C$ such that $\nu_{n}(a\otimes 1_{C})=\iota_{n}(a)\otimes 1_{C}$ and $\nu_{n}(1_{A_n}\otimes C)\subseteq (A_{n+1})^{\cC}_{\iota_n}\otimes C$. Then each $\nu_{n}$ carries an equivariant structure with respect to the actions $\cC\overset{\widetilde{F}_{n}}{\curvearrowright} A_{n}\otimes C$. In particular, there exists an action $\cC\overset{\widetilde{F}}{\curvearrowright} \varinjlim A_{n}\otimes C$.
\end{cor}

\subsection{Extending simple stationary actions}

Given a unitary fusion category $\mathcal{C}$ and a right finitely semisimple $\mathcal{C}$-module category $\mathcal{M}$ (for definitions see \cite{MR3687214}), recall that the \textit{dual category} $\mathcal{C}^{*}_{\mathcal{M}}=\text{End}_{\mathcal{C}}(\mathcal{M})^{mp}$ is the monoidal opposite of the (multi)-fusion category of $\cC$-module endofunctors on $\mathcal{M}$. For example, the ``trivial" $\cC$-module category ($\cC$-acting on itself by right tensoring), the dual category is $\mathcal{C}$ itself.

\begin{cons}\label{AF-construction} \textbf{Stationary AF-actions}. Given a unitary fusion category $\mathcal{C}$ with tensor product denoted $\boxtimes$, a right finitely semisimple $\mathcal{C}$-module category $\mathcal{M}$, an object $X\in \mathcal{C}^{*}_{\mathcal{M}}$ and a choice of $m\in \mathcal{M}$ containing every isomorphism class of simple in $\mathcal{M}$ as a summand, we can define an AF-action of $\mathcal{C}$ as follows (see also \cite[Section 5]{evington2024equivariant}):

Set $L_{n}:=X^{\boxtimes n}\triangleright m\in \mathcal{M}$. Define the finite dimensional C*-algebras $A_{n}:=\text{End}_{\mathcal{M}}(L_{n})$. Then we have an action $F_{n}:\mathcal{C}\rightarrow \text{Bim}(A_{n})$ given by

$$F_{n}(Y):=\text{Hom}_{\mathcal{M}}(L_{n}, L_{n}\otimes Y)$$

This naturally has the structure of a (right) $A_{n}$ correspondence, with left and right actions given by 

$$a\triangleright f \triangleleft b:=(a\boxtimes 1_{Y})\circ f\circ b $$
and right $A_{n}$ valued inner product

$$\langle f | g\rangle:=f^{*}\circ g$$

The tensorators are easy to write down. Note this is simply the action of $\mathcal{C}$ on $\mathcal{M}\cong \text{Mod}(\text{End}(X^{n}\boxtimes m))$, translated to the bimodule formalism (as in Remark \ref{finitedimensionalaction}). The algebra connecting maps $\iota_{n}:A_{n}\rightarrow A_{n+1}$ are defined by $\iota_{n}(a):=1_{X}\otimes a$, and the bimodule connecting maps $j^{Y}_{n}: F_{n}(Y)\rightarrow F_{n+1}(Y)$ are defined by $j^{Y}_{n}(f):=1_{X}\otimes f$.

AF-actions built this way are called \textit{stationary AF-actions}. If in addition the module category $\mathcal{M}$ is indecomposable and the fusion matrix $X\otimes \cdot$ has no non-zero entries we call the stationary action \textit{simple}.
\end{cons}

\begin{rem}
The assumption that $X$ is a strong tensor generator guarantees that the AF algebra $A$ is simple and has a unique trace. The indecomposability condition is mostly for convenience, since it ensures $C^{*}_{\mathcal{M}}$ is unitary fusion rather than unitary multi-fusion. This assumption can almost certainly be removed with care, but we find it convenient to have a simple unit in $C^{*}_{\mathcal{M}}$ for the proofs below.
\end{rem}

\begin{thm}\label{mainextension} Let $\mathcal{C}$ be a unitary fusion category and let $\cC\overset{F}{\curvearrowright} A$ be a simple stationary AF-action with corresponding data $(\mathcal{M},m, X)$. Then if $X$ contains $\mathbbm{1}$ with multiplicity $l>1$ in $\mathcal{C}^{*}_{M}$, there exists a simple, unital AT C*-algebra $B$ with unique trace, $K_{0}(B)\cong K_{0}(A)$, $K_{1}(B)\cong \mathbbm{Z}^{\text{rank}(\mathcal{M})}$ and an $A\mathbbm{T}$-action $\cC\overset{F}{\curvearrowright} B$.
\end{thm}

\begin{proof}

Given the simple stationary data $(\mathcal{M},m, X)$, let $A_{n}$, $F_{n}$, $\iota_{n}$ and $j^{Y}_{n}$ be as above. The resulting AF algebra $A$ is simple with unique trace. Set $B_{n}:= C(\mathbbm{T})\otimes A_{n}$. Our goal is to define homomorphisms $\nu_{n}:B_{n}\rightarrow B_{n+1}$ which satisfy the hypotheses of Lemma \ref{mainlemma} with respect to this data, in such a way that the resulting inductive limit algebra $B=\varinjlim B_{n}$ satisfies the desired conclusion of the theorem.

To simplify notation, let $\mathcal{D}:=\mathcal{C}^{*}_{M}$.
By definition, we have the algebra $\text{End}_{\mathcal{D}}(X)\boxtimes 1_{L_{n}}\subseteq A^{\mathcal{C}}_{n+1}$. For each simple object $x\in \text{Irr}(\mathcal{D})$, let $Z_{x}\in \text{End}_{\mathcal{D}}(X)$ be the central projection onto the simple summand corresponding to $x$. Pick an orthonormal basis $\{e_{i}\}^{l}_{i=1}$ for the space $\text{Hom}_{\mathcal{D}}(\mathbbm{1}, X)$ with respect to the composition pairing. Note $l>1$ by hypothesis. Then the set $\{e_{i}\circ e^{*}_{j}\}$ is a set of matrix units for the matrix summand of $\text{End}_{\mathcal{D}}(X)$ corresponding to the simple object $\mathbbm{1}$.

Now, define the unitary in $C(\mathbbm{T})\otimes \text{End}_{\mathcal{D}}(X)$

$$W(z):=z\otimes (e_{1}\circ e^{*}_{n})+ \sum^{l}_{i=1}1_{C(\mathbbm{T})}\otimes (e_{i}\circ e^{*}_{i-1})\ \ +\sum_{x\in \text{Irr} (\mathcal{D})-\{\mathbbm{1}\}} 1_{C(\mathbbm{T})} \otimes Z_{x}$$

We define $\nu_{n}:C(\mathbbm{T})\otimes A_{n}\rightarrow C(\mathbbm{T})\otimes A_{n+1}$ on simple tensors by

$$\nu_{n}(f(z)\otimes a):=f(W(z))\boxtimes a\in C(\mathbbm{T})\otimes \text{End}_{\mathcal{D}}(X)\boxtimes \text{End}_{\mathcal{M}}(L_{m})\subseteq B_{n+1},$$

\noindent where by $f(W(z))$ we mean in the sense of functional calculus. We are also making use of the canonical identification $C(\mathbbm{T})\otimes D\cong C(\mathbbm{T}, D)$ for any unital C*-algebra $D$, given on simple tensors by 

\begin{align}\label{iso}
a=\sum_{n\in \mathbbm{Z}} z^{n}\otimes a_{n}\mapsto \widetilde{a}(z)
\end{align}
where $\widetilde{a}(\omega):=\sum_{n\in \mathbbm{Z}}\omega^{n}a_{n}$.

The map $\nu_{n}$ we have defined extends to an injective $*$-homomorphism $\nu_{n}: B_{n}\rightarrow B_{n+1}$. Furthermore, by construction $\nu_{n}|_{1_{C(\mathbbm{T})}\otimes A_{n}}=\iota_{n}$. Since $ \text{End}_{\mathcal{H}^{*}}(X)\boxtimes 1_{L_{n}}\subset (A_{n+1})^{\mathcal{C}}_{\iota_{n}}$, we also have $\nu_{n}(C(\mathbbm{T})\otimes 1_{L_n})\subseteq C(\mathbbm{T})\otimes (A_{n+1})^{\mathcal{C}}_{\iota_{n}} $, and thus the conditions of Lemma \ref{mainlemma} are satisfied for each $n$. By Corollary \ref{maincor}, this data assembles into an inductive limit action $\cC\overset{\widetilde{F}}{\curvearrowright} B$, where $B=\lim B_{n}$.

Now we compute the K-theory of $B$. Note that the map $A_{n}\rightarrow  C(\mathbbm{T})\otimes A_{n}=B_{n}$ given by $a\mapsto 1_{C(\mathbbm{T})}\otimes a$ induces an (order) isomorphism on $K_{0}$. Thus we have that the commuting diagram

\[
\begin{tikzcd}
A_{1}
\arrow[swap]{d}{\cdot \otimes 1}
\arrow{r}{\iota_{1}}
& A_{2}
\arrow{r}{\iota_{2}}\arrow[swap]{d}{\cdot \otimes 1}
& A_{3} \arrow[swap]{d}{\cdot \otimes 1} \arrow{r}{\iota_{3}}
& \dots
 \\
B_{1}
\arrow{r}{\nu_{1}}
& B_{2}
\arrow{r}{\nu_{2}}
& B_{3}\arrow{r}{\nu_3}
&\dots
\end{tikzcd}
\]

\noindent induces a commuting diagram of (ordered) $K_{0}$ groups such that the vertical arrows are isomorphisms. Thus $K_{0}(A)$ is order isomorphic to $K_{0}(B)$

To understand $K_{1}$, we recall that for the algebra $$D:=C(\mathbbm{T})\otimes (M_{m_1}(\mathbbm{C})\oplus M_{m_2}(\mathbbm{C})\oplus \dots \oplus M_{m_t}(\mathbbm{C})),$$ we have $K_{1}(D)\cong \mathbbm{Z}^{t}$. The generating unitaries of the t copies of $\mathbbm{Z}$ are given by

$$u_{i}:=I_{m_1}\oplus \dots \oplus \left[ {\begin{array}{cc}
   z & 0 \\
   0 & I_{m_i-1} \\
  \end{array} } \right]\oplus\dots \oplus I_{m_t}$$

  \medskip

\noindent Applying this to the algebras $B_{n}=C(\mathbbm{T})\otimes A_{n}$, we denote the generating unitaries in $B_{n}$ by $u_{n,i}$, where $1\le i\le m$, and $m=\text{rank}(\mathcal{M})$. Then by construction of $\nu_{n}$, we see $l [\nu_{n}(u_{n,i})]_{1}=[\nu_{n,i}(u_{n,i})^{l}]_{1}=l [u_{n+1,i}]_{1}$, since each $K_{1}$ is a free abelian group. Thus $K_{1}(\nu_{n})([u_{n,i}]_{1})=[u_{n+1,i}]_{1}$, hence under the above indentification of $K_{1}(B_{n})\cong \mathbbm{Z}^{t}$, we see that the $\nu_{n}$ maps induce the identity map on $K_{1}$. This implies that we obtain an isomorphism in the limit $K_{1}(B)\cong \mathbbm{Z}^{t}$.

Now we will show that $B$ is simple, using essentially the original argument for simplicity of Bunce-Deddens algebras \cite{MR0365157}.

Let $\{P_{i,n}\}^{m}_{i=1}$ denote the minimal central projections in $B_{n}$ corresponding to the objects $n_{1},\dots, n_{m} $ respectively. Suppose $I\subseteq B$ is an ideal. Then $I\cap B_{n}$ is an ideal of $B_{n}$, and $I=\varinjlim I\cap B_{n}$. Suppose for some $n$ that $P_{i,n}(I\cap B_{n})\ne 0$. Note $P_{i,n}(I\cap B_{n})=I\cap P_{i,n}B_{n}$ is an ideal in $P_{i,n}B_{n}$. Utilizing the isomorphism from \ref{iso}, there exists a closed subset $T_{i,n}\subseteq \mathbbm{T}$ such that $I\cap P_{i,n}B_{n}=\{a\in B_{n}\ : a=P_{i,n}a,\ \text{and}\ \widetilde{a}(\omega)=0\ \text{for all}\ \omega\in T_{i,n}\}$.

Our claim is that $T_{i,n+1}\ne \varnothing$ and if $t\in T_{i,n+1}$, then all the $l^{th}$ roots of $t$ are in $T_{i,n}$. For $a=\sum_{n\in \mathbbm{Z}} z^{n}\otimes a_{n}\in B_{n}$ with $P_{i,n}a=a$, for $t\in T_{i,n+1}$ we have

$$\widetilde{P_{i,n+1}\iota_{n}(a)}(t)=0,$$
which is equivalent to the Fourier expansion

$$P_{i,n+1}(\sum_{n\in \mathbbm{Z}} W(t)^{n}\boxtimes a_{n})=0.$$ 

But $W(t)=1\otimes t(e_{1}\circ e^{*}_{n})+ \sum^{l}_{i=1}1_{C(\mathbbm{T})}\otimes (e_{i}\circ e^{*}_{i-1})\ \ +\sum_{x\in \text{Irr}(\mathcal{D})-\{\mathbbm{1}\}} 1_{C(\mathbbm{T})} \otimes Z_{x}$ has eigenvalues consisting of all $l^{th}$ roots of $t$, $\{t_{1},\dots t_{l}\}$, and possibly $1$ with some multiplicity. But by construction all of the $t_i$ have components that survive cutting down by $P_{n+1,i}$ (this is because the only non-trivial terms live over tensoring with $\mathbbm{1}$ in $\mathcal{D}$). This implies 

$$\widetilde{a}(t_{i})=0.$$

Thus $t_{i}\in T_{i,n}$.
Since $P_{i,n}(B_{n}\cap I)\ne 0$ implies $P_{i,n+1}(B_{n+1}\cap I)\ne 0$ since any $\iota_{n}(a) $ will have non-zero component in $P_{i,n+1}(B_{n+1}\cap I)$ for any $a\in P_{i,n}(B_{n}\cap I)$. 
We now claim $T_{i,n+1}\ne \varnothing.$ If it were, then $P_{i,n+1}(I\cap B_{n+1})=P_{i,n+1}B_{n+1}\ne 0$. But this implies $P_{i,n+1}A_{n+1}\subseteq I\cap A_{n+1}$, and thus $I\cap A$ is a non-zero ideal in $A$, contradicting simplicity of the AF algebra $A$.

Continuing inductively, we see $T_{i,n+k}\ne 0$ for all $k$, and for every $s_{k}\in T_{i,n+k}$, all the $(l^{k})^{th}$ roots of $s_{k}$ are in $T_{i,n}$. But for any sequence $\{s_{k}\}^{\infty}_{k=1}\subseteq \mathbbm{T}$, the set $\cup_{k} \{(l^{k})^{th}\ \text{roots of}\ s_{k}\}\subseteq T_{i,n}$ is dense in $\mathbbm{T}$, and since $T_{i,n}$ is closed, we have equality. This contradicts $P_{i,n}I\cap B_{n}=0$. Thus $B$ is simple.

It remains only to show that $B$ has a unique trace. Note that $A$ has a unique trace, and thus any trace on $B$ must restrict to the unique trace on $A$. For each $B_{n}$, any trace is determine by a ``trace vector"

$$\overrightarrow{\tau}_{n}:=(\lambda_{1,n}\mu_{1,n},\dots, \lambda_{m,n}\mu_{m,n})$$

where $\mu_{i,n}$ is a state on $C(\mathbbm{T})$, and $\lambda_{i,n}$ are positive scalars with $\sum_{i}\lambda_{i,n} \sqrt{dim(P_{i,n}A_{n})}=1$. Then the corresponding state on $B_{n}$ is given by

$\tau_{n}(f(z)\otimes x):=\sum \lambda_{i,n}\mu_{i,n}(f(z)) \text{Tr}_{i,n}(P_{i,n}x)$.

Note that the uniqueness of trace on $A$ implies the $\lambda_{i,n}$ are completely determined. We will show that compatibility of the trace with the connecting maps $\iota_{n}$ totally determines the $\mu_{i,n}$. We compute

$$\tau_{n+1}(\iota_{n}(P_{i,n}z^{k})):=\begin{cases}
\sum_{x\in \text{Irr}(\mathcal{D})-\{\mathbbm{1}\}} \tau_{n+1}(P_{i,n}\otimes Z_{x}) & \ k\ne0\ \text{mod l}\\
\tau_{n+1}(P_{i,n}\otimes Z_{\mathbbm{1}})\mu_{i,n+1}(z^{\frac{k}{l}})+\sum_{x\in \text{Irr}(\mathcal{D})-\{\mathbbm{1}\}} \tau_{n+1}(P_{i,n}\otimes Z_{x}) & k=0\ \text{mod l}
\end{cases}$$

Now, since $P_{i,n}\otimes Z_{\mathbbm{1}}\in A_{n+1}$ and $\tau_{n+1}$ is completely determine on $A$, the values $\tau_{n+1}(P_{i,n}\otimes Z_{\mathbbm{1}})$ are fixed. But we have the consistency condition 

$$\tau_{n+1}(\iota_{n}(P_{i,n}z^{k}))=\tau_{n}(P_{i,n}z^{k})=\tau_{n}(P_{i,n})\mu_{i,n}(z^{k}),$$

\noindent and since $\tau_{n}(P_{i,n})\ne 0$, we have $$\mu_{i,n}(z^{k})=\frac{\tau_{n+1}(\iota_{n}(P_{i,n}z^{k}))}{\tau_{n}(P_{i,n})}.$$

Thus if $k\ne 0\ \text{mod l}$, the value of $\mu_{i,n}(z^{k})$ is completely determined by $\tau|_{A}$. If $k= 0\ \text{mod l}$, let $d$ be the largest power of $l$ that divides $k$. Then $\mu_{i,n}(z^{k})$ is determined by $\mu_{i,n+1}(z^{\frac{k}{l}})$, which inductively is determined by $\mu_{i,n+d}(z^{\frac{k}{l^{d}}})$ which, by the above argument is determined by $\tau|_{A}$ since $\frac{k}{l^{d}}\ne 0\ \text{mod l}$. Therefore, there is at most one trace $\tau$ on $B$. Since $B$ is stably finite and nuclear, it admits a trace (for example, \cite{MR3241179}).

\end{proof}

\begin{remark}\label{rem:differentactions}
If the fusion matrix for $X$ on $\mathcal{M}$ is in $GL_{\text{rank}(\mathcal{M})}(\mathbbm{Z})$, then $K_{0}(A)\cong \mathbbm{Z}^{\text{rank}(\mathcal{M})}$ as abelian groups. Then one could take the AF-action of $\mathcal{C}$ on the AF-algebra $A$ built from the above data and simply tensor with the classifiable C*-algebra $D$ with $K_{0}\cong K_{1}\cong \mathbbm{Z}$ with unique trace. Then by the Kunneth product formula (\cite{MR894590}) and the classification theorem for simple nuclear C*-algebras \cite{2015arXiv150100135G,2015arXiv150703437E}, the algebra $B\cong D\otimes A$. In particular, the tensor extension of the action on $A$ to $A\otimes D$ yields an action of $\mathcal{C}$ on $B$, which has the nice regularity property of being manifestly $Z$-stable in the sense of \cite{evington2024equivariant}. However, it was communicated to us by George Elliott that $D$ is \textit{not} itself AT, so it is not clear whether the action constructed this way is an AT-action. This raises a question that could serve as a first goal for any attempt to classify fusion category actions on C*-algebras: are the actions on the AT-algebras constructed via the above theorem above equivalent to the ones obtained by tensoring the underlying AF-action by $D$?

In general, however, Theorem \ref{mainextension} produces actions on AT-algebras $B$ which are not obviously isomorphic to $A\otimes D$ for \textit{any} AF-algebra $A$ and any classifiable C*-algebra $D$. 
\end{remark}

\section{Actions on noncommutative tori}

Noncommutative tori are a family of noncommutative algebras that have many interesting properties and applications. As mentioned in the introduction, they are the prototype for Connes' noncommutative geometry \cite{MR1303779} and have applications in both high energy \cite{MR1613978} and condensed matter physics \cite{MR1295473}. Noncommutative spaces come in several flavors. For example, we can consider algebraic, smooth and topological noncommutative spaces. We briefly recall all three versions for the noncommutative torus here.

First we consider the algebraic version. Let $\Theta$ be an $n\times n$ skew-symmetric real matrix. Then define $A^{0}_{\Theta}$ to be the unital, associative $*$-algebra generated by unitaries $U_{1}, \dots U_{n}$ satisfying the defining relations

$$U_{i}U_{j}=e^{2\pi i \Theta_{ji}}U_{j}U_{i}$$

This can be viewed as a twisted group algebra $\mathbbm{C}[\mathbbm{Z}^{n}]$ by a 2-cocycle defined by $\Theta$. $A^{0}_{\Theta}$ is a very natural multiplicative version of the Weyl algebra, and is often conceptualized as the algebra of \textit{polynomial} functions on a noncommutative torus.

The \textit{topological version} of the noncommutative torus $A_{\Theta}$ is the universal C* generated by unitaries satisfying the defining relations above. For any fixed $\Theta$, we have inclusions

$$A^{0}_{\Theta}\subseteq A_{\Theta}$$

The matrix $\Theta$ is said to be \textit{degenerate} if there exists an $x\in \mathbbm{Q}^{n}$ such that $\langle x, \Theta y\rangle\in \mathbbm{Q}$ for all $y\in \mathbbm{Q}^{n}$. Then by \cite{phillips2006simple}, if $\Theta$ is non-degenerate then $A_{\Theta}$ is a simple AT C*-algebra with unique trace and real rank $0$, which gives us access to Elliot's classification theorem \cite{MR1241132}. This allows us to build C*-algebras isomorphic to $B$ by constructing a simple AT algebra  with the appropriate K-theory and traces.

As abelian groups, $K_{0}(A_{\Theta})\cong K_{1}(A_{\Theta})\cong \mathbbm{Z}^{2^{n-1}}$. Furthermore, in the simple case the order structure on $K_{0}(A_{\Theta})$ is completely determined by the value of the unique trace $\tau$, i.e. $[p]-[q]\in K_{0}(A)$ is positive if and only if $\tau(p)-\tau(q)>0$ (or the class itself represents $0$). The range of the trace in $\mathbbm{R}$ is in general somewhat complicated, but is determined in \cite{MR0679700}. 

In this paper, the goal is to construct actions of interesting fusion categories on the topological noncommutative torus $A_{\Theta}$. 

\subsection{No-go results}\label{sec:No-go}.
We now give a series of no-go results that suggest the existence of an action of a fusion category on \textit{any} noncommutative torus is non-trivial.

\begin{thm}
Let $\mathcal{C}$ be a fusion category containing an object with non-integral dimension, and let $A^{0}_{\Theta}$ be an algebraic noncommutative torus of any dimenison. There is no linear monoidal functor $\mathcal{C}\rightarrow \text{Bim}(A^{0}_{\Theta})$\footnote{in this context, Bim denotes the linear monoidal category of purely algebraic bimodules with relative tensor product}.
\end{thm}

\begin{proof}
Any algebraic action on $A_{\Theta}$ by definition yields a linear monoidal functor $\mathcal{C}\rightarrow \text{Bim}(A^{0}_{\Theta})$. But $K_{0}(A^{0}_{\Theta})\cong \mathbbm{Z}$ (see, for example \cite[Theorem 7.1]{MR1604198}). Thus if we had an action of $\mathcal{C}$ on $A^{0}_{\Theta}$, this would give a unital ring homomorphism from the fusion ring $\text{Fus}[\mathcal{C}]\rightarrow \text{End}(\mathbbm{Z})\cong \mathbbm{Z}$. This extends to a character on the complexified fusion algebra which sends $[\mathbbm{1}]$ to $1$, and thus by \cite[Proposition 3.3.6]{MR3242743}, this must be the Frobenius-Perron dimenison $\text{FPdim}$, contradicting the existence of an object with non-integral dimension.
\end{proof}

The obstruction above is due to the fact that $K_{0}$ is $\mathbbm{Z}$ for the algebraic noncommutative torus, or in other words, every projective module is stably free. This is not the case for the topological noncommutative torus, and thus we have hope of building interesting fusion category actions in that setting. Nevertheless, we have the following no-go theorem, which shows that certain fusion categories admit \textit{no topological action} whatsoever on any noncommutative torus of any rank.

\begin{prop} Suppose $\mathcal{C}$ is a unitary fusion category whose dimensions of simple objects are rationally independent. Then if $\text{rank}(\mathcal{C})$ is not a power of 2, there is no action of $\mathcal{C}$ on any noncommutative torus.
\end{prop}

\begin{proof}
By hypothesis, the dimension function $\text{FPdim}: \mathbbm{Q}[\mathcal{C}]\rightarrow \mathbbm{R}$ is a ring isomorphism onto its image. But since the set $\text{Irr}(\mathcal{C})$ forms a basis for the rational fusion ring $\mathbbm{Q}[\mathcal{C}]$ and the dimensions of these objects are algebraic integers \cite[Proposition 3.3.4]{MR3242743}, $\mathbbm{Q}[\mathcal{C}]$ is a field. Therefore $K:=\mathbbm{Q}[\mathcal{C}]$ is a field, and $[K:\mathbbm{Q}]=\text{rank}(\mathcal{C})$. Then, if $\mathcal{C}$ had an action on the noncommutative torus $A_{\Theta}$ of rank $n$, this would induce a $\mathbbm{Z}[\mathcal{C}]$-module structure on $K_{0}(A_{\Theta})\cong \mathbbm{Z}^{2^{n-1}}$. Extending scalars gives a $\mathbbm{Q}[\mathcal{C}]$-module structure on $K_{0}(A_{\Theta})\otimes_{\mathbbm{Z}} \mathbbm{Q}\cong \mathbbm{Q}^{2^{n-1}}$. But as pointed out above, $\mathbbm{Q}[\mathcal{C}]$ is a field, so this module is free, i.e. as $\mathbbm{Q}[\mathcal{C}]$-modules, 

$$\mathbbm{Q}^{2^{n-1}}\cong \mathbbm{Q}[\mathcal{C}]^{m}$$

\noindent for some $m$. This implies $2^{n-1}=m\cdot \text{rank}(\mathcal{C})$, so $\text{rank}(\mathcal{C})$ divides $2^{n-1}$.
\end{proof}

The condition that the dimensions of simple objects are rationally independent is very strong. We know of only one family of fusion categories satisfying this condition. Recall $\text{PSU}(2)_{k}$ for $k$ odd is a unitary fusion category with $\frac{k+1}{2}$ simple objects  $\{Y_{a}\}_{0\le a\le \frac{k-1}{2}}$ and fusion rules 

$$Y_{a}\otimes Y_{b}\cong \oplus N^{c}_{ab}Y_{c}$$

$$N^{c}_{ab}=\delta_{|a-b|\le c\le \text{min}\{l-(a+b),a+b\}}$$

\noindent If we set $q=e^{\frac{\pi i}{k+2}}$ and $[n]_{q}=\frac{q^{n}-q^{-n}}{q-q^{-1}}$, then we have $\text{dim}(Y_{a})=[2a+1]_{q}$.

\bigskip

The following lemma was communicated to us by Andrew Schopieray:

\begin{lem}\label{lem:rationalind} Suppose $k>1$ and $k+2$ is prime. Then the dimensions of simple objects of $\text{PSU}(2)_{k}$ are rationally independent.
\end{lem}

\begin{proof}
Since $q=e^{\frac{\pi i}{k+2}}$, $[m]_{q}=[k+2-m]_{q}$ for $m=0,1,\dots \frac{k+1}{2}$. This implies that the set of dimensions

$$\{\text{dim}(Y_{a})\}_{0\le a\le \frac{k-1}{2}}=\{[m]_{q}\}_{1\le m\le \frac{k+1}{2}}.$$

If this set is linearly dependent over $\mathbbm{Q}$, after clearing denominators, we could find integers $a_{m}$ such that

$$\sum^{\frac{k+1}{2}}_{m=1} a_m [m]_{q}=0.$$

Multiplying by $q-q^{-1}$ gives

$$\sum^{\frac{k+1}{2}}_{m=1} a_m (q^{m}-q^{-m})=\sum^{\frac{k+1}{2}}_{m=1} a_m (q^{m}-q^{2(k+2)-m})=0.$$

The result now follows, for example see \cite{244563}.

\end{proof}

\begin{cor}
Let $k>1$ such that $k+2$ is prime and $k+1\ne 2^{n}$. Then any unitary fusion category containing $\text{PSU}(2)_{k}$ as a full subcategory admits no topological action on a noncommutative torus of any rank. In particular, for these values of $k$, $\text{SU}(2)_{k}$ do not admit actions on noncommutative tori.
\end{cor}

We will see in a later section examples of actions of $\text{PSU}(2)_{15}$ on noncommutative tori. Note $15+2$ is prime, but $15+1=2^{4}$.

\subsection{Actions of Haagerup-Izumi on noncommutative 2-tori}\label{sec:HaagIz}

Let $G$ be an abelian group. The Haagerup-Izumi fusion ring over $G$ has a set of isomorphism classes of simple objects indexed by $G\cup\{g\rho\}_{g\in G}$, where we identify $e\rho=\rho$ for the unit $e\in G$, and $\mathbbm{1}=e$. The fusion rules are determined by

$$g\otimes h\cong gh\,, \quad g\otimes \rho \cong g\rho\,, \quad \rho\otimes g \cong g^{-1}\rho\,,$$

$$\rho\otimes \rho\cong \left(\bigoplus_{g\in G} g\rho \right) \oplus \mathbbm{1}$$

In this paper, by a Haagerup-Izumi fusion category over an abelian group $G$ we will mean any unitary categorification of the above described Haagerup-Izumi fusion ring \textit{with the additional condition that the subcategory of invertible objects is equivalent to $\Hilb(G)$}. In particular, we assume the cohomology class of the associator restricted to the pointed part is trivial. In Haagerup-Izumi categories, there are exactly $|G|$ objects with dimension $1$ and $|G|$ objects with dimension $\frac{|G|+\sqrt{|G|^{2}+4}}{2}$.

Haagerup-Izumi fusion categories were first introduced as a family by Izumi \cite{MR1832764} generalizing Haagerup's original example (corresponding to the case $G=\mathbbm{Z}/3\mathbbm{Z}$)  constructed with Asaeda in the context of subfactors \cite{MR1686551}, after a proof of existence was previously announced in \cite{MR1317352}. In \cite{MR3827808} these categories are studied extensively under the name ``generalized Haagerup categories". It is not currently known whether there are infinitely many abelian groups $G$ such that the Haagerup-Izumi fusion ring has a unitary categorification \cite{MR2837122}.

In this section, we will construct actions of Haagerup-Izumi categories (assuming they exist) on noncommutative 2-tori. An important observation for us is that any Haagerup-Izumi category has a canonical rank 2 module category, which we describe below.

\begin{prop}
Let $\mathcal{H}$ be a Haagerup-Izumi category over an abelian group $G$. Then there exists a canoncial indecomposable rank 2 (right) $\mathcal{H}$-module category $\cM$. The dual category $\mathcal{H}^{*}$ has the following properties:

\begin{enumerate}
\item 
$\mathcal{H}^{*}$ contains a copy of $\text{Hilb}(\widehat{G})$, and each $\gamma\in \widehat{G} $ acts trivially as a functor on $\mathcal{M}$
\item 
There is a simple object $\pi\in \mathcal{H}^{*}$ with dimension $\frac{|G|+\sqrt{|G|^{2}+4}}{2}$ whose fusion matrix for $\mathcal{M}$ is the same as $\rho$.
\item 
The isomorphism classes of simple objects are $\widehat{G}\cup \widehat{G}\pi$, hence $\mathcal{H}^{*}$ has rank $2|G|$.

\end{enumerate}
\end{prop}
\begin{proof}

Let $\mathcal{H}$ be a Haagerup-Izumi fusion category over $G$. Let $A=\mathbb{C}[G]\in \text{Hilb}(G)\le \mathcal{H}$ be the standard group algebra Q-system. Then we define the (right) $\mathcal{H}$-module category $\mathcal{M}:=_{A} \mathcal{H}$ to be the category of left $A$ modules. The dual category is denoted $\mathcal{H}^{*}$, which is equivalent to the fusion  category of $A$-$A$ bimodules, $_{A} \mathcal{H}_{A}$.

First we claim that $\cM$ is rank $2$. Indeed, the functors $g\otimes \cdot  :\mathcal{H}\rightarrow \mathcal{H}$ assemble into a categorical action of $G$ on $\mathcal{H}$ (the latter is viewed as a linear category, not a fusion category). Then we have an equivalence of categories $_{A} \mathcal{H}\cong \mathcal{H}^{G}$, where the superscript denotes the equivariantization. But simple objects in the equivariantization are classified up to isomorphism by orbits and projective representations of the stabilizer subgroup of a point in the orbit, where the 2-cohomology class governing the projective representations is derived from the $G$-action \cite{MR3059899}. In our case, this $G$ action has 2-orbits, both of which have trivial stabilizer subgroups. 

Thus there are precisely 2 isomorphsim classes of simple objects in $\cM$, which are described explicitly as left $A$ modules by $m_{1}:= A\otimes \rho\cong \bigoplus_{g\in G}g\rho$ and $m_{2}:=A=\bigoplus_{g\in G} g$ , with the obvious left $A$-module structures. Now notice that each invertible $g\in \mathcal{H}$ acts trivially (as a functor up to natural isomorphism) on $\cM$. The fusion rules with $\rho$ are thus

$$m_{1}\otimes \rho \cong m_{2}\oplus m^{\oplus |G|}_{1}$$
$$m_{2}\otimes \rho\cong m_{1}$$

In matrix form, we have $m\mapsto m\otimes \rho$ represented by the matrix \[
  \left[ {\begin{array}{cc}
   |G| & 1 \\
   1 & 0 \\
  \end{array} } \right]
\]

Recall that there is a canonical ``dual" Q-system $B\in H^{*}$ such that $H^{*}_{B}\cong _{A}H$ and $_{B}H^{*}_{B}\cong H$. Indeed, realizing $H^{*}\cong _{A} H_{A}$, then $B\cong A\otimes A$ with the obvious $A$-$A$ bimodule structure, with algebra map induced by the evaluation cup between middle factors (this is simply the ``basic construction" Q-system from subfactor theory). 
For a general group $G$, if $A$ is the group algebra Q-system then $_{A}\text{Hilb}(G)_{A}\cong \text{Rep}(G)$, and $B\cong \text{Fun}(G)$ (the latter being the algebra of functions on $G$ with pointwise multiplication). If $G$ is abelian, then $\text{Rep}(G)\cong \text{Hilb}(\widehat{G})$, and under any choice of this isomorphism, $\text{Fun}(G)\cong \mathbbm{C}[\widehat{G}]$.

Thus we have that $\text{Hilb}(\widehat{G})\le \mathcal{H}^{*}$ is a fusion subcategory, and the Q-system $B$, dual to $A$, lies in this subcategory and is isomorphic to the group algebra $\mathbbm{C}[\widehat{G}]$. But since $\mathcal{H}^{*}_{B}\cong _{A} \mathcal{H}$ has rank 2, there are precisely 2 simple right $B$ modules. But by the logic we used above, this implies tensoring on the right with $\widehat{G}$ has precisely two orbits, and each orbit has no stabilizer subgroups. Thus we can write the (isomorphism classes of) simple objects of $\mathcal{H}^{*}$ as 

$$\widehat{G}\cup \{\pi \gamma\}_{\gamma\in \widehat{G}}$$

\noindent where $\pi$ is some object in the orbit not containing $\mathbbm{1}$. Fix such a $\pi$. Notice that by construction $\gamma\in \widehat{G}$ satisfies $\gamma \otimes m_{2}\cong m_{2}$, and since the $\gamma$ are invertible, they must act by equivalences on $\cM$ so that $\gamma \otimes m_{1}\cong m_{1}$. This implies that each $\pi \gamma$ will have the same fusion rules with respect to $m_{1}$ and $m_{2}$ as any $\pi \gamma^{\prime}$.

Now for any object in $\mathcal{H}^{*}$, it's fusion matrix on $\cM$ must commute with \[
  \left[ {\begin{array}{cc}
   |G| & 1 \\
   1 & 0 \\
  \end{array} } \right],
\] hence is of the form

\[
  \left[ {\begin{array}{cc}
   a+b|G| & b \\
   b & a \\
  \end{array} } \right]=a\left[ {\begin{array}{cc}
   1 & 0 \\
   0 & 1 \\
  \end{array} } \right]+ b\left[ {\begin{array}{cc}
   |G| & 1 \\
   1 & 0 \\
  \end{array} } \right]
\]
for $a$ and $c$ non-negative integers. Note that if an object $\pi\in \mathcal{H}^{*}$ has matrix representation with $a$ and $b$ as above, then the quantum dimension is
$d(\pi)=a+b\frac{|G|+\sqrt{|G|^{2}+4}}{2}$. Let $\phi=\frac{|G|+\sqrt{|G|^{2}+4}}{2}$. Now, since $\mathcal{H}$ and $\mathcal{H}^{*}$ have the same global dimension (they are Morita equivalent) we have

$$|G|+\sum_{g\in G} d(g\rho)^{2}=|\widehat{G}|+\sum_{\gamma\in \widehat{G}}d(\pi \gamma)^{2}$$

And since $G$ is abelian and each object in the non-trivial orbit has the same dimension, this reduces to

$$\phi^{2}=(a+b\phi)^{2}$$

If $b=0$, and we have $\phi^{2}=a^{2}$, but $\phi^{2}=|G|\phi+1$, and since $\phi$ is irrational this is a contradiction. Thus $b\ne 0$, but then we must have $a=0$ and $b=1$. In particular, the fusion matrix for $\pi \otimes$ has the required form.

\end{proof}

\begin{remark}
Note that $\mathcal{H}^{*}$ looks suspiciously like a Haagerup-Izumi catgeory. Indeed, all known examples where $\mathcal{H}^{*}$ has been explicitly determined are themselves Haagerup-Izumi categories over $\widehat{G}\cong G$. If $G$ is cyclic or $|G|$ is odd, then corresponding dual Haagerup-Izumi categories are explicitly identified as Haagerup-Izumi categories \cite{MR3827808}.
\end{remark}

Now we turn to building actions of Haagerup-Izumi categories on noncommutative 2-tori. For any finite abelian group $G$, set $\phi_{G}:=\frac{|G|+\sqrt{|G|^{2}+4}}{2}$. Note this is always irrational. Define

$$\Theta_{G}:=\left[ {\begin{array}{cc}
   0 & \phi_{G} \\
   -\phi_{G} & 0 \\
  \end{array} } \right]$$

\begin{thm}\label{Haagerupaction}
For any Haagerup-Izumi category $\mathcal{H}$ over the abelian group $G$, there is an action on the noncommutative 2-torus $A_{\Theta_{G}}$.
\end{thm}

\begin{proof}
Pick the rank $2$ $\mathcal{H}$-module category $\mathcal{M}$ we just described, and choose $X=\pi^{\otimes 4}$. Choose $m=m_{1}\oplus m_{2}$, and let $\mathcal{M}$ be the module category in question. Then $(\mathcal{M},m,X)$ satisfies  the hypothesis of Theorem \ref{mainextension}. Theorefore we have an action of $\mathcal{H}$ on a simple AT C*-algebra $B$ with unique trace and $K_{1}(B)\cong \mathbbm{Z}^{2}$, and $K_{0}(B)\cong K_{0}(C)$, where $C$ is the corresponding AF-algebra. But from \cite[Chapter 6]{MR0623762}, $K_{0}(C)\cong \mathbbm{Z}+\mathbbm{Z}[\phi]\cong K_{0}(A_{\Theta_{G}})$ as ordered abelian groups, but the isomorphism takes the order unit in $[1_{C}]\in K_{0}(A)$ to $|G|$ in $\mathbbm{Z}+\mathbbm{Z}[\phi]$. Thus we obtain $B\cong M_{|G|}(A_{\Theta_{G}})$.

\end{proof}

\begin{thm}
For any Haagerup-Izumi category $\mathcal{H}$ over the abelian group $G$, there is an action of the Drinfeld center $\mathcal{Z}(\mathcal{H})$ on the noncommutative 2-torus $A_{\Theta_{G}}$. 
\end{thm}

\begin{proof}
    One obvious way to do this would be simply to compose the canonical forgetful functor $\mathcal{Z}(\mathcal{H})\rightarrow \mathcal{H}$ and then apply the action constructed in the previous theorem. However, this action will not be fully faithful. Instead, notice that $\mathcal{M}$ as above is a rank 2 module category for $\mathcal{Z}(\mathcal{H})$, whose dual category is $\mathcal{H}^{mp}\boxtimes \mathcal{H}^{*}$. Choose the object $\rho^{mp}\boxtimes \pi$, which generates the dual category. Then since $\rho^{mp}\boxtimes \pi$ has the same fusion graph as $\pi^{2}$, the argument from Theorem \ref{Haagerupaction} applies to give a nice action of $\mathcal{Z}(\mathcal{H})$ on $A_{\Theta_{G}}$.
\end{proof}

\subsection{Action of adjoint subcategory of $E_{8}$ quantum subgroup on a noncommutative 3-torus}\label{sec:AdE8}

There are two unitary fusion categories which arise as the adjoint subcategory of the $E_{8}$ quantum subgroup of $\text{SU}(2)$ (or $E_{8}$ subfactor) which have the same fusion ring and are related by complex conjugation \cite{MR1313457}. In this section, we will refer to either of these categories by $\mathcal{E}$.

Then $\mathcal{E}$ has four isomorphism classes of simple objects $\{\mathbbm{1}, A, B, \tau\}$ satisfying the following  fusion rules \cite{MR1284945}:

$$A^{2}=\mathbbm{1}\oplus A\oplus B$$
$$AB=BA=A\oplus 2B\oplus \tau$$
$$A\tau=\tau A= B$$
$$B^{2}=\mathbbm{1}\oplus 2A\oplus 3B\oplus \tau$$
$$B\tau=\tau B=A\oplus B$$
$$\tau^{2}=\mathbbm{1}\oplus \tau$$

The subcategory $\{\mathbbm{1},\tau\}$ is equivalent to the unitary $\textbf{Fib}$ category.

We have $$\alpha:=\text{FPdim}(A)=\frac{1}{4}(3+\sqrt{5}+\sqrt{6(5+\sqrt{5})})$$

$$\beta:=\text{FPdim}(B)=\frac{1}{2}(2+\sqrt{5}+\sqrt{3(5+2\sqrt{5})})$$

$$\phi:=\text{FPdim}(\tau)=\frac{1+\sqrt{5}}{2}$$
which have minimal polynomials $$p_{\alpha}(x)=x^{4}-3x^{3}-x^{2}+3x+1,\ p_{\beta}(x)=x^{4}-4x^{3}-4x^{2}+x+1,\  p_{\phi}(x)=x^{2}-x-1$$

We consider $\mathcal{M}=\mathcal{E}$ as the indecomposable right $\mathcal{E}$-module category, so that $\mathcal{E}^{*}_{\mathcal{M}}\cong \mathcal{E}$. Set $m$ to be the sum over simples, and choose $X=A^{\otimes 4}$. Then the triple $(\mathcal{M},m,X)$ satisfies the hypotheses of Theorem \ref{mainextension}, and thus we have an action on a simple AT-algebra $B$ with unique trace with $K_{1}(B)\cong \mathbbm{Z}^{4}$.

With the natural ordering $m_{1}=\mathbbm{1}, m_{2}=A, m_{3}=B$ and $m_{4}=\tau$, the fusion matrix  for $X \otimes {\cdot}$ (and hence the connecting diagram for $\iota_{n}$) is $$\left[\begin{array}{cccc}
   3  & 7 & 10 & 3 \\
    7 & 20 & 30 & 10\\
    10 & 30 & 50 & 17\\
    3 & 10 & 17 & 6
\end{array}\right]\in GL(4,\mathbbm{Z})^{+}$$

In particular, each inclusion $\iota_{n}$ is an isomorphism on $K_{0}$, and hence $K_{0}(B)\cong \mathbbm{Z}^{4}$ as an abelian group.
Recall that the order structure on $K_{0}(B)$ is completely determined by the trace, in the sense that $K^{+}_{0}(B)=\{[p]-[q]\ : \tau(p)-\tau(q)>0\}\cup \{0\}$. We have the following theorem, which will give us a nice basis with which to compute the trace pairing with $K_{0}$.

\begin{lem}\label{basislemma}
    There exists a set of finitely generated projective Hilbert modules $\{H_{C}\}_{C\in \text{Irr}(\mathcal{E})}$ of $B$ which generate $K_{0}(B)$ and satisfy $\tau[H_{C}]=\text{dim}(C)$.
\end{lem}

\begin{proof}
    Consider the inductive limit action $F:\mathcal{E}\rightarrow \rCorr(B)$, and set $H_{C}:=B\boxtimes_{B}F(C)\cong F(C)$ as a right Hilbert $B$ module.

Then since $\tau$ induces the unique state on $K_{0}$, it must be the case that $\tau([H_{C}])=\tau([B_{B}\boxtimes_{B}F(C)])=\tau([B_{B}])\text{dim}(C)=\text{dim}(C)$ \cite[Proposition 5.2]{MR4419534}.

Thus it suffices to show that $\{[H_{C}]\}$ forms a basis for $\mathbbm{Z}^{4}$. Let $H^{1}_{C}=F_{1}(C)$ as a right $B_{1}=\mathbbm{C}^{4}\otimes C(\mathbbm{T})$ module. Then $[H_{C}]$ is the image of  $[H^{1}_{C}]$ under the inductive limit inclusions, all of which are isomorphism. Hence if the $\{H^{1}_{C}\}$ are a basis if and only in the $\{H_{C}\}$ are. But using the fusion rules we compute

$$[H^{1}_{\mathbbm{1}}]=\left[\begin{array}{c}
1\\
1\\
1\\
1
\end{array}\right],\ [H^{1}_{A}]=\left[\begin{array}{c}
1\\
3\\
4\\
1
\end{array}\right],\ [H^{1}_{B}]=\left[\begin{array}{c}
1\\
4\\
7\\
2
\end{array}\right],\ 
[H^{1}_{\tau}]=\left[\begin{array}{c}
1\\
1\\
2\\
2
\end{array}\right]
$$

but we see that the matrix

$$\left[\begin{array}{cccc}
   1 & 1 & 1 & 1 \\
    1 & 3 & 4 & 1\\
    1 & 4 & 7 & 2\\
    1 & 1 & 2 & 2
\end{array}\right]\in GL(4,\mathbbm{Z})^{+}$$

\bigskip

Thus $\{[H^{1}_{C}]\}_{C\in \text{Irr}(\mathcal{E})}$ form a basis for $K_{0}(B_{1})\cong \mathbbm{Z}^{4}$.

\end{proof}

Now, consider the skew symmetric matrix $\Theta:=\left[\begin{array}{ccc}
   0 & \phi & \alpha \\
    -\phi & 0 & \beta \\
    -\alpha & \beta & 0     
\end{array}\right]$

\bigskip

Then let $A_{\Theta}$ be the corresponding noncommutative $3$-torus.

\begin{lem}
$\Theta$ is non-degenerate, hence $A_{\Theta}$ is a simple A$\mathbbm{T}$ algebra with unique trace.
\end{lem}

\begin{proof}

Suppose $x=\left[\begin{array}{c}
a\\
b\\
c
\end{array}\right]\in \mathbbm{Q}^{3}$ satisfies that property that for all $y\in \mathbbm{Q}^{3}$, 
$\langle x, \Theta y\rangle\in \mathbbm{Q}$.

Then 

$$\left[\begin{array}{c}
-b\phi-c\alpha\\
a\phi-\beta c\\
a\alpha+b\beta
\end{array}\right]\cdot y\in \mathbbm{Q}$$

\bigskip

for all $y\in \mathbbm{Q}^{3}$. Choosing $y=\left[\begin{array}{c}
1\\
0\\
0
\end{array}\right],\left[\begin{array}{c}
0\\
1\\
0
\end{array}\right]$ respectively yields

\bigskip

$$-b\phi-c\alpha, a\phi-\beta c\in \mathbbm{Q}.$$

But $\phi$ is quadratic and $\alpha,\beta$ are quartic, hence the sets $\{1,\alpha,\phi\}$ and $\{1,\beta,\phi\}$ are rationally independent. This implies $a=b=c=0$, so $\Theta$ is non-degenerate.
\end{proof}

\begin{cor}
$B\cong A_{\Theta}$, and thus we have an action of $\mathcal{E}$ on a noncommutative 3-torus. 
\end{cor}

\begin{proof}
    $B$ and $A_{\Theta}$ are both simple AT algebras with unique trace, with both $K_{0}$ and $K_{1}$ isomorphic to $\mathbbm{Z}^{4}$. It remains to show that $K_{0}(B)$ and $K_{0}(A_{\Theta})$ are order isomorphic. But in both cases, the order is determined by the traces, and there is a basis for $\mathbbm{Z}^{4}$ whose trace values are precisely the $\Theta_{i,j}$ (for $B$, by our lemma, for $A_{\Theta}$ by \cite[Section 4.3]{MR4102570}.). The map sending basis elements to their counterpart with the same trace value is thus an order isomorphism.
\end{proof}

\subsection{Action of $\text{PSU}(2)_{15}$ on a noncommutative 4-torus}\label{sec:PSU(2)}

\bigskip

Let $\mathcal{F}:=PSU(2)_{15}$, where we use notations and conventions as described in Section \ref{sec:No-go}. This fusion category is rank $8$, and we consider $\mathcal{M}:=\mathcal{F}$ as a right $\mathcal{F}$-module category, and we set $m=\sum_{a} Y_{a}$. Note $\mathcal{F}^{*}_{\mathcal{M}}\cong \mathcal{F}$. Then set $X=Y^{\otimes 4}_{1}$. The triple $(\mathcal{M},m,X)$ satisfies the hypothesis of Theorem \ref{mainextension}, and thus we obtain an action of $\mathcal{F}$ on the simple A$\mathbbm{T}$ $B$ with unique trace and $K_{1}(B)\cong \mathbbm{Z}^{8}$.

To compute $K_{0}$, the fusion matrix for $X\otimes \cdot $ is the $8\times 8$ matrix

$$\left[\begin{array}{cccccccc}
   3 & 6 & 6 & 3 & 1 & 0 & 0 & 0 \\
   6 & 15 & 15 & 10 & 4 & 1 & 0 & 0\\
   6 & 15 & 19 & 16 & 10 & 4 & 1 & 0 \\
    3 & 10 & 16 & 19 & 16 & 10 & 4 & 1 \\
     1 & 4 & 10 & 16 & 19 & 16 & 10 & 1 \\
      0 & 1 & 4 & 10 & 16 & 19 & 16 & 9 \\
      0 & 0 & 1 & 4 & 10 & 16 & 18 & 12  \\
     0 & 0 & 0 & 1 & 4 & 9 & 12 & 9 \\    
\end{array}\right]\in GL(8,\mathbbm{Z})^{+}$$

\bigskip

In particular, since the connecting maps of the Bratteli diagram are $K_{0}$ isomorphisms, this implies $K_{0}(B)\cong \mathbbm{Z}^{8}$. We now prove an analogue of Lemma \ref{basislemma}.

\begin{lem}
There exists a set of finitely generated projective Hilbert modules $\{H_{a}\}_{0\le a\le 7}$ of $B$ which generate $K_{0}(B)$ and satisfy $\tau[H_{a}]=\text{FPdim}(Y_{a})$.

\end{lem}

\begin{proof}
The proof follows closely Lemma \ref{basislemma}.
Let $F:\mathcal{F}\rightarrow \rCorr(B)$ constructed in \ref{mainextension}.  Then since $B$ has unique trace $\tau$, setting $H_{a}=B\boxtimes_{B} F(Y_{a})$ we have $\tau([H_{a}])=\text{FPdim}(Y_{a})$. Now to show these are generators, since the connecting maps in the Bratteli diagram for $B$ are isomorphisms on $K_{0}$, it suffices to show $H^{1}_{a}=B_{1}\boxtimes_{B} F_{1}(Y_{a})$ are a basis for $K_{0}(B_1)$. But $B_{1}=\mathbbm{C}^{8}$, with the minimal projections indexed by the simple objects $Y_{a}$. These form a basis for $K_{0}(B_{1})\cong \mathbbm{Z}^{8}$ and with the obvious order from the index $a$, we compute 

$$[H^{1}_{0}]=\left[\begin{array}{c}
1\\
1\\
1\\
1\\
1\\
1\\
1\\
1
\end{array}\right],\ [H^{1}_{1}]= \left[\begin{array}{c}
1\\
3\\
3\\
3\\
3\\
3\\
3\\
2
\end{array}\right],\ [H^{1}_{2}]= \left[\begin{array}{c}
1\\
3\\
5\\
5\\
5\\
5\\
4\\
2
\end{array}\right],\ [H^{1}_{3}]= \left[\begin{array}{c}
1\\
3\\
5\\
7\\
7\\
6\\
4\\
2
\end{array}\right],\ [H^{1}_{4}]= \left[\begin{array}{c}
1\\
3\\
5\\
7\\
8\\
6\\
4\\
2
\end{array}\right] $$

$$ [H^{1}_{5}]= \left[\begin{array}{c}
1\\
3\\
5\\
6\\
6\\
6\\
4\\
2
\end{array}\right],\ [H^{1}_{6}]= \left[\begin{array}{c}
1\\
3\\
4\\
4\\
4\\
4\\
4\\
2
\end{array}\right],\ [H^{1}_{7}]= \left[\begin{array}{c}
1\\
2\\
2\\
2\\
2\\
2\\
2\\
2
\end{array}\right] $$

\bigskip

The matrix whose columns are these vectors has determinant $1$, so is in $GL(8,\mathbbm{Z})^{+}$. Thus $\{[H_{a}]\}_{0\le a\le 7}$ form a basis for $K_{0}(B)$.

\end{proof}

\begin{cor}
   $K_{0}(B)\cong \mathbbm{Z}[\text{dim}(Y_{0}),\dots \text{dim}(Y_{7})]\subset \mathbbm{R}$ as partially ordered abelian groups
\end{cor}

\begin{proof}
By the previous lemma $\tau: K_{0}(B)\rightarrow \mathbbm{Z}[\text{dim}(Y_{0}),\dots \text{dim}(Y_{7})]$ is surjective. But by Lemma \ref{lem:rationalind} this is an isomorphism.  
\end{proof}

\bigskip Our next goal is to identify the $B$ constructed above with a rank $4$ noncommutative torus.

With $q=e^{\frac{i\pi}{17}}$ and $[n]_{q}$ as above, define

$$\Theta:=\left[\begin{array}{cccc}
   0 & [2]_{q} & [8]_{q} & [14]_{q} \\
    -[2]_{q} & 0 & [6]_{q} & [10]_{q} \\
    -[8]_{q} & -[6]_{q} & 0  & [4]_{q}\\
    -[14]_{q} & -[10]_{q} & -[4]_{q} & 0
\end{array}\right]$$.

\begin{thm} $B\cong A_{\Theta}$, and thus $\mathcal{F}$ admits an action on a noncommutative $4$-torus.
\end{thm}

\begin{proof}
By Lemma \ref{lem:rationalind} $\Theta$ is non-degenerate, so $A_{\Theta}$ is simple with unique trace $\tau$. $K_{0}(A_{\Theta})\cong K_{1}(A_{\Theta})\cong \mathbbm{Z}^{8}$, which agrees with $B$. It remains to show $K_{0}(A_{\Theta})$ is order isomorphic to $K_{0}(B)$. But since $\Theta$ is totally irrational, $\tau:K_{0}(A_{\Theta})\rightarrow \mathbbm{R}$ is an isomorphism onto its image. Furthermore, by \cite{MR0679700,MR4102570}, we have 

$$\tau(K_{0}(A))= \mathbbm{Z}[[2]_{q}, [4]_{q}, [6]_{q}, [8]_{q}, [10]_{q}, [14]_{q}, \text{pf}(\Theta)]$$
where $\text{pf}(\Theta)=[2]_{q}[4]_{q}-[8]_{q}[10]_{q}+[6]_{q}[14]_{q}$, which an easy computation shows is $-[12]_{q}$. Thus we have $\tau(K_{0}(A))= \mathbbm{Z}[[2]_{q}, [4]_{q}, [6]_{q}, [8]_{q}, [10]_{q}, [14]_{q}, \text{pf}(\Theta)]$, hence as partially ordered abelian groups

$$K_{0}(A)\cong \mathbbm{Z}[\text{dim}(Y_{0}),\dots \text{dim}(Y_{7})]\cong K_0(B)$$ 
where the last isomorphism is from the above corollary.
\end{proof}

\begin{footnotesize}
\noindent{\it Acknowledgements.}
This work was initiated during the authors visits to the Hausdorff Research Institute in Bonn in August 2022 for the follow up workshop 
to the earlier  programme on von Neumann algebras. 
DEE is supported by an Emeritus Fellowship from the Leverhulme Trust. Part of the work was conducted when he visited
Dublin Institute for Advanced Studies, the University of New South Wales and Harvard University. He would like to thank Denjoe O’Connor, Pinhas Grossman (and the Australian Research Council Discovery Project DP170103265) and Arthur Jaffe for their generous hospitality and stimulating environments. CJ is supported by NSF Grant DMS-2247202.
We are grateful to Alexander Betz, George Elliott, Sergio Gir\'on Pacheco, and Andrew Schopieray for valuable comments. We also thank Kan Kitamura for pointing out an error in a previous version of the manuscript. 
\end{footnotesize}

\bibliographystyle{alpha}
{\footnotesize{
\bibliography{bibliography}
\end{document}